\documentclass[a4paper, 12pt]{amsart}

\usepackage{amssymb,amsmath}
\usepackage[dvips]{graphics}
\usepackage[dvips]{color}
\usepackage{graphicx}
\usepackage[english]{babel}
\usepackage[latin1]{inputenc}
\usepackage{comment}

\newtheorem{theorem}{Theorem}[section]
\newtheorem{lemma}[theorem]{Lemma}

\newtheorem{corollary}[theorem]{Corollary}

\theoremstyle{definition}
\newtheorem{definition}[theorem]{Definition}

\newtheorem{remark}[theorem]{Remark}

\numberwithin{equation}{section}

\newcommand{\be}{\begin{equation}}
\newcommand{\ee}{\end{equation}}
\newcommand{\beq}{\begin{eqnarray}}
\newcommand{\eeq}{\end{eqnarray}}
\newcommand{\beqy}{\begin{eqnarray*}}
\newcommand{\eeqy}{\end{eqnarray*}}

 \def\s{{\sigma}}

\newcommand{\bt}{\begin{theorem}}
\newcommand{\et}{\end{theorem}}
\newcommand{\bp}{\begin{prop}}
\newcommand{\ep}{\end{prop}}
\newcommand{\bc}{\begin{corollary}}
\newcommand{\ec}{\end{corollary}}
\newcommand{\bl}{\begin{lem}}
\newcommand{\el}{\end{lem}}
\newcommand{\br}{\begin{rem}}
\newcommand{\er}{\end{rem}}
\newcommand{\beg}{\begin{exam}}
\newcommand{\eeg}{\end{exam}}

\newcommand{\bit}{\begin{itemize}}
\newcommand{\eit}{\end{itemize}}
\newcommand{\bin}{\begin{enumerate}}
\newcommand{\ein}{\end{enumerate}}

\newcommand{\g}{\gamma}
\newcommand{\e}{\epsilon}
\newcommand{\oo}{\infty}
\title{The Riesz Capacity in Metric Spaces}

\author{Juho Nuutinen and Pilar Silvestre}

\newcommand\diam{\operatorname{diam}}
\newcommand\dist{\operatorname{d}}

\providecommand{\ch}[1]{\text{\raise 2pt \hbox{$\chi$}\kern-0.2pt}_{#1}}

\providecommand{\vint}[1]{\mathchoice
          {\mathop{\vrule width 5pt height 3 pt depth -2.5pt
                  \kern -9pt \kern 1pt\intop}\nolimits_{\kern -5pt{#1}}}%
          {\mathop{\vrule width 5pt height 3 pt depth -2.6pt
                  \kern -6pt \intop}\nolimits_{\kern -3pt{#1}}}%
          {\mathop{\vrule width 5pt height 3 pt depth -2.6pt
                  \kern -6pt \intop}\nolimits_{\kern -3pt{#1}}}%
          {\mathop{\vrule width 5pt height 3 pt depth -2.6pt
                  \kern -6pt \intop}\nolimits_{\kern -3pt{#1}}}}

\begin{document}

\thanks{The research was supported by the Academy of Finland, grant no. 272886}

\begin{abstract}
We study a capacity theory based on a definition of a Riesz potential in metric spaces with a doubling measure. In this general setting, we study the basic properties of the Riesz capacity, including monotonicity, countable subadditivity and several convergence results. We define a modified version of the Hausdorff measure and provide lower bound and upper bound estimates for the capacity in terms of the modified Hausdorff content. 
%We also study isocapacitary inequalities and boundedness of the Riesz potential.
%We also study the boundedness of the Riesz potential in $L^p$-spaces and characterize it by an isocapacitary inequality. 
\end{abstract}

\subjclass[2010]{31E05, 31B15}

%\date{\today}  
\maketitle

\section{Introduction}

In this paper, we study a theory of capacity based on a metric version of the Riesz potential in the setting of a general metric space $(X, d)$ equipped with a doubling measure $\mu$. We define a related Hausdorff measure and study the connections between the Riesz capacity and the Hausdorff measure. 
%\textbf{Part of the motivation to work with this particular definition of the Riesz potential comes from the work of Kinnunen, Silvestre: Resistance conditions and applications? %With our definition we avoid the assumption of Ahlfors $Q$-regularity. Laakso: For any $Q \geq 1$ there exists an Ahlfors $Q$-regular space admitting a weak 
%$(1,1)$-Poincare inequality.} 
%We study the relations between the capacity and a modified Hausdorff measure that works in a general metric space. 
With our definitions and results, we extend the classical Riesz capacity theory from the Euclidean space, with the Lebesgue measure, to the setting of a general metric measure space.
%by only assuming the doubling condition \eqref{doubling measure} from the measure $\mu$ throughout the paper. 
In $\mathbb{R}^n$, the capacity theory for the Riesz potential can be found for example in \cite{AH}, \cite{AE}, \cite{Me} and \cite{Mi}.
%We end the paper by studying the boundedness of the Riesz potential in $L^p$-spaces and Lorentz spaces.
%This leads us to define a modified version of the standard Hausdorff measure. 
During the past twenty years, different capacities have been studied in metric measure spaces for example in \cite{BB}, \cite{GT}, \cite{HaKi}, \cite{HKST}, \cite{KS}, \cite{KKST}, \cite{KM}, \cite{KM2} and \cite{L}. Also, a part of the theory for Riesz capacity follows from general results in \cite{Fu} and \cite{Sj2}. Here, we formulate the theory explicitly and state the results to keep the paper self-contained. 
%In our work, we extend these results and give direct proofs to more general results. 
%However, for the bigger parts, the theory of the Riesz capacity in a metric measure space is new, and in our generality completely new. Especially, when %studying the connections between the Riesz capacity and Hausdorff measure, we give direct proofs to results that are more general and apply not only for %compact sets.

We define a metric version of the Riesz potential of order $\gamma$, where $0<\gamma<1$, as $$I_{\gamma} f(x) = \int_{X} \frac{f(y)}{\mu\left(B(x,d(x,y))\right)^{1-\gamma}} \, d\mu (y).$$ One can find a similar definition for the Riesz potential in the works of Kairema and Sj\"odin (see \cite{Ka}, \cite{Sj} and \cite{Sj2}). In the definition, there appears only the measure of balls in the Riesz kernel. Another definition for a metric version of the Riesz potential is such that it also has the distance function as a part of the kernel. This version of the Riesz potential can be found for example in \cite{HK}, \cite{He} and \cite{Ma1}. Also, other Riesz potentials and fractional integral operators have been studied in the metric setting for example in \cite{GV}, \cite{GSV} and \cite{Ma2}. We emphasize that, throughout the paper, we do not assume any type of (Ahlfors) $Q$-regularity on the measure $\mu$ that would give uniform lower bounds or upper bounds for the measure of balls in terms of the radii. In this generality, our definition of the Riesz potential, with no distance function as a part of the kernel, works better.
%In particular, when studying the relations between the capacity and the Hausdorff measure.

In Section 3, we define a metric version of the Riesz capacity $\mathcal{C}_{\gamma,p}$ and show that it satisfies the basic properties of capacity. These properties include monotonicity, countable subadditivity and several convergence results. In particular, we show that the Riesz capacity is a Fatou capacity. This lower semicontinuity property of capacity is an analogue of Fatou's lemma. We also study the capacitability of sets and show that the Riesz capacity is a so called Choquet capacity. This means that the capacity of a Borel set can be obtained by approximating with compact sets from the inside and open sets from the outside. We finish the section by briefly studying the dual Riesz capacity.

%The main motivation to work with the definition of Riesz potential, which only has measure of a ball in the kernel, comes from Section 4. With this %definition, assuming only the doubling condition on the measure $\mu$, 
In the beginning of Section 4, we prove an upper bound estimate for the capacity of balls in terms of the measure $\mu$. This result leads us to define a modified version of the standard Hausdorff content. The main results of the section are lower bound and upper bound estimates for the Riesz capacity in terms of this modified Hausdorff content.
%We provide two results that relate the Riesz capacity and the Hausdorff measure. 
%More precisely, we prove lower and upper bound estimates for the Riesz capacity of a set $E \subset X$ in terms of the Hausdorff measure. 
Similar results have been studied by Sj\"odin in \cite{Sj}. 
%As a special case, our results give the main result there (see \cite[Theorem 2.2]{Sj}). 
Here, we give direct proofs to results that apply not only for compact sets. In particular, we do not need to use Frostman's lemma to obtain the results.  %We stress the fact that 
%%%%%%%%%%%%%%%%%%%%%%%without assuming completeness or locally compactness from the space $X$, we prove 
%our results are more general and apply not only for compact sets. Also, our proofs are direct and in particular, we don't need to use Frostman's lemma to %obtain the results.

%Finally, in Section 5, we study isocapacitary inequalities and boundedness of the Riesz potential. In this section, we assume our measure $\mu$ to be a Radon %measure. In \cite{Ka}, it is shown that our metric version of the Riesz potential is bounded as an operator from $L^p(X)$ to $L^q(X)$, if and only if a certain restriction %is placed on the order $\gamma$ of the Riesz potential $I_\gamma$. This is an analogous result in the metric case for the well-known Hardy-Littlewood-Sobolev %theorem in the Euclidean space. In Theorem 5.2, we give three equivalent conditions to the boundedness of $I_\gamma$ from $L^p(X)$ to $L^q(X)$. In particular, an %isocapacitary inequality characterizes the boundedness of the Riesz potential. Similar result in the Euclidean space can be found in the paper by Adams %(see~\cite{A2}). However, in the metric setting our results are completely new. We also prove a lemma, which shows that the Riesz potential is bounded from %$L^p(X)$ to a Lorentz capacitary space $L^{p,q}(\mathcal{C}_{\gamma,p})$.

\textbf{Acknowledgements.} We would like to thank Professor Juha Kinnunen for proposing this project. We would also like to thank Juha Lehrb\"ack and Heli Tuominen for useful discussions and comments on the manuscript.

\section{Notation and preliminaries}

\subsection{Riesz potential}

We assume that $X=(X,d,\mu)$ is a locally compact metric measure space
equipped with a metric $d$ and a Borel regular,  doubling outer
measure $\mu$. The doubling property means that there is a fixed constant
$c_d\geq 1$, called the doubling constant of $\mu$, such that 
\begin{equation}\label{doubling measure}
\mu(B(x,2r))\le c_d\mu(B(x,r))
\end{equation}
for every ball $B(x,r)=\{y\in X:\dist(y,x)<r\}$. We also assume that the measure of each open ball is positive and finite. The doubling condition implies that
\begin{equation}\label{doubling dimension}
\frac{\mu(B(y,r))}{\mu(B(x,R))}\ge C\Big(\frac rR\Big)^Q 
\end{equation}
for every $0<r\le R$ and $y\in B(x,R)$ for some $C>0$ and $Q>0$ that only depend on $c_d$. 
In fact, we may take $Q=\log_2c_d$ and $C=c_d^{-2}$ (see~\cite{BB}). In addition, we assume that spheres are of measure zero, i.e. 
\begin{equation}\label{spheres}
\mu \left( \left\{y \in X: d(x,y) = r\right\}\right)= 0,
\end{equation}
for $x \in X$ and $B(x,r)$. This assumption is needed for the Riesz potential, defined below, to satisfy lower semicontinuity properties (see Remark 3.3) that are required for the capacity theory.
\begin{definition}
Let $0<\gamma<1$. The Riesz potential of order $\gamma$ of a measurable function $f$  is
\begin{equation}\label{our definition}
I_{\gamma} f(x) = \int_{X} \frac{f(y)}{\mu\left(B(x,d(x,y))\right)^{1-\gamma}} \, d\mu (y).
\end{equation}
\end{definition}
\begin{remark}
%%(i) In the definition, we consider the balls to be closed. This is important for the lower semicontinuity of the Riesz potential in the space $X$. We show %%in Remark 3.3 that the Riesz potential, as defined above, is lower semicontinuous. \\
(i) To be precise, we would need to define the kernel separately for the cases $x\neq y$ and $x=y$. However, we assume our space $X$ to be such that $\mu$ vanishes on sets which consist of a single point. Then the domain of integration $X\setminus \{x\}$ can be replaced by $X$ (see~\cite{Ka}). Since we have a doubling metric measure space, this is equivalent to the condition that there are no isolated points in our space $X$. \\
(ii) In the Euclidean space, with the $n$-dimensional Lebesgue measure, we have, with the notation $\alpha = \gamma n \in (0,n)$, the usual Riesz potential $$I_\alpha f(x)= \int_{\mathbb{R}^n} \frac{f(y)}{|x-y|^{n-\alpha}} \, dy$$ of order $\alpha$ on $\mathbb{R}^n$ (up to a dimensional constant). 
\end{remark}
Another way to define a Riesz potential in a metric space, as in~\cite{HK} and \cite{He}, is
\begin{equation}\label{other definition}
\widetilde I_{\gamma} f(x) = \int_{X} \frac{f(y) d(x,y)^{\gamma}}{\mu\left(B(x,d(x,y))\right)} \, d\mu (y).
\end{equation}
If the measure $\mu$ is (Ahlfors) $Q$-regular, that is, there exists a constant $C>1$ such that 
\begin{equation}\label{ahlfors}
C^{-1}r^Q\le\mu(B(x,r))\le Cr^Q
\end{equation}
for every $x\in X$ and $0<r<\diam(X)$, then $I_\gamma f$ and $\widetilde I_{\gamma Q} f$ are comparable
in the sense that there exists a constant $C\ge1$ 
such that
\[
C^{-1} I_\gamma f\le\widetilde I_{\gamma Q} f\le C I_\gamma f.
\]
%%%%%%In case only the lower bound holds in the Alhfors regularity condition,
%%%%%%then we say that the measure satisfies the measure lower bound condition.

In the next sections, we do not assume the (Ahlfors) $Q$-regularity or any other estimates that would give uniform lower bounds or upper bounds for the measure of balls in terms of the radii. We assume only the doubling property \eqref{doubling measure} and develop the theory of Riesz capacity based on the definition \eqref{our definition} of the Riesz potential. In particular, this definition works better for our purposes in Section 4, where we define a modified version of the standard Hausdorff measure and prove two results that relate the Riesz capacity and the Hausdorff measure. 

\subsection{Function spaces and capacities}
We have by Cavalieri's principle that $L^p(X)=L^p(X,\mu)$ is the space of all $\mu$-measurable functions $f$ in $X$ such that 
$$ \|f\|_{L^p(X)} =\left( \int_{0}^{\oo} p t^{p-1} \mu(\{z \in X: |f(z)|>t\}) dt \right)^{1/p} < \oo,$$
which is a Banach space when $1 \leq p < \infty$. The weak $L^p$-space $L^{p,\oo}(X)$ is defined by the condition
$$\|f\|_{L^{p, \oo}(X)} := \sup_{t > 0 } t \mu\left(\{z \in X: |f(z)| > t\}\right)^{{1}/{p}}<\oo.$$ We denote by $L^p_{+}(X)$ the subset of $L^p(X)$ of non-negative functions. 
%\textbf{Onko} Since $(X, d, \mu)$ is a $\sigma$-finite measure space and $\|f\|_{L^p(X)} = \| g\|_{L^p(X)}$, when
%$$ \mu(\{z \in X: |f(z)|>t\}) = \mu(\{z \in X: |g(z)|>t\}) \quad \forall \, \, t >0,$$ we say that $L^p(X)$ is a rearrangement invariant function %space (see~\cite{BS}). \textbf{tarpeen?}
%Therefore, for each finite value of $t$ belonging to the range of $\mu$, let $E$ be a subset of $X$ with $\mu(E)=t$ and let 
%$$ \varphi_{L^p(X)}(t) = \|\chi_{E}\|_{L^p(X)}.$$
%The function $\varphi_{L^p(X)}$ is called the fundamental function of $L^p(X)$, and we will denoted through the paper by $\varphi_{p}$. Observe that %the particular choice of the set $E$ with $\mu(E)=t$ is immaterial. 
%\textbf{Mik\"a olisi paras tapa m\"a\"aritell\"a kapasiteetti meid\"an tapauksessa? (vrt. esim. Meyers)}
\begin{definition}
We define, on the family of $\mu$-measurable subsets of $X$, a capacity to be a non-negative set function $\mathcal{C}$, which has the following properties: 
\begin{itemize}
\item[(a)] $\mathcal{C}(\emptyset)=0$,
\item[(b)] If $A \subset B$, then $\mathcal{C}(A) \leq \mathcal{C}(B)$, 
\item[(c)] $\mathcal{C}\left(\bigcup_{i=1}^{\infty} A_i \right) \leq \sum_{i=1}^{\infty} \mathcal{C}(A_i)$. \end{itemize} 
\end{definition}
%In this case, $(\O, \S, \mathcal{C})$ is called a capacity space. 
%If moreover for all measurable sets $A$ and $B$ on $\S$ $$ \mathcal{C}(A \cup
%B) \leq C(\mathcal{C}(A) + \mathcal{C}(B)),$$ where $C \geq 1$ is a constant, we say
%that the capacity is quasi-subadditive; it is subadditive if $C= 1$. 
%The last property is called quasi-subadditivity and if it holds for $c=1$, we say that the capacity is subadditive. 
A capacity $\mathcal{C}$ is called a Fatou capacity if $\mathcal{C}( A_{i}) \to \mathcal{C}(A)$, whenever $A_1 \subset A_2 \subset \cdots$ are subsets of $X$ and $A=\bigcup_{i=1}^\infty A_i$. We also say that a property holds $\mathcal{C}$-q.e.\ on $X$ if it holds for all $x \in X$ except those in a set $E$ with $\mathcal{C}(E)=0$.

The capacitary Lorentz spaces $L^{p,q}(\mathcal{C})$, $p, q>0$, are
defined by the condition $$\| f \|_{L^{p,q}(\mathcal{C})} := \Big(q \int_{0}^{\infty} t^{q-1} \mathcal{C}\left(\{z \in X: |f(z)|> t\}\right)^{{q}/{p}} dt \Big)^{{1}/{q}}<\oo,$$ when $q<\oo$, and in the case of $q=\infty$ by 
$$\| f \|_{L^{p,\infty}(\mathcal{C})} :=\sup_{t > 0 } t \, \mathcal{C}\left(\{z \in X: |f(z)| > t\}\right)^{{1}/{p}}<\oo.$$ The space $L^{p,\oo}(\mathcal{C})$ is called the weak capacitary $L^p$-space. For the general facts and properties of the capacitary Lorentz spaces, we refer to \cite{BS}, \cite{C} and \cite{CMS}.
%Observe that for $p, q >0$, $ \| f \|_{L^{p,q}(\mathcal{C})} = 0$ if
%and only if $f = 0$ \rm \, $\mathcal{C}$-q.e.\ and equivalent functions have
%the same $ \| \cdot \|_{L^{p,q}(\mathcal{C})}$-quasi-norm. Moreover, for every
%$\l \in \R$, $ \| \lambda f \|_{L^{p,q}(\mathcal{C})} = |\lambda | \|  f
%\|_{L^{p,q}(\mathcal{C})}$ and $$ \| f + g \|_{L^{p,q}(\mathcal{C})} \leq 2C( \| f
%\|_{L^{p,q}(\mathcal{C})} + \| g \|_{L^{p,q}(\mathcal{C})}) .$$ Thus $L^{p,q}(\mathcal{C})$
%is a quasi-normed function space. 
%\textbf{Siirra} It follows easily that
%$L^{p,q}(\mathcal{C})$, $0<p <\oo$, $0 < q \leq \oo$ is a quasi-normed lattice on the space of measurable functions and $L^{p,q}(\mathcal{C})$ is
%complete. \textbf{5. lukuun? M\"a\"arittele lattice.} 
Throughout the paper, we denote the characteristic function of a set $E\subset X$ by $\ch{E}$.  
In general, $C$ will denote a positive constant whose value is not
necessarily the same at each occurrence.  
%%%%%%
%The integral average of a function $f\in L^1(A)$ over a
%$\mu$-measurable set $A$ with finite and positive measure is 
%denoted by 
%\[
%f_A=\vint{A}f\,d\mu=\frac1{\mu(A)}\int_Af\,d\mu.
%\]
%%%%%%%%%%

\section{Riesz Capacity}
\begin{definition}
Let $1<p<\infty$ and $0<\gamma<1$. The Riesz $(\gamma,p)$-capacity of a set $E \subset X$ is the number
$$\mathcal{C}_{\gamma,p}(E)= \inf_{f \in \mathcal{A}(E)} ||f||^{p}_{L^p(X)},$$
where $$\mathcal{A}(E)= \left\{f \in L^p_{+}(X): I_\gamma f \geq 1 \text{ on } E \right\} .$$
\end{definition}
If $\mathcal{A}(E)= \emptyset$, we set $\mathcal{C}_{\gamma,p}(E)= \infty$. Functions belonging to $\mathcal{A}(E)$ are called admissible functions or test functions for $E$. From now on, we always assume in this section that $1<p<\infty$ and $0<\gamma<1$.
%and if a property holds for all $x \in X$ except those in a set $E$ with $\mathcal{C}_{\gamma,p}(E)=0$, we say that it holds %$\mathcal{C}_{\gamma,p}$-q.e. on $X$.
%%%%%%%%%%%%%%%%%%%%%%%%%%%%%%%%%%%%%%%%%%%%
%\textbf{Onko} We can see easily (see~\cite{Me}) that 
%\begin{equation}\label{E1}
%\mathcal{C}_{\g,p}(E)= (\sup_{f}\big\{  \inf_{x \in E} I_{\g} f(x) : f \in L^p_{+}(X), \| f\|_{L^p(X)} \leq 1 \big\})^{-p}.
%\end{equation}
%\textbf{tarpeen?}

In the Euclidean space, with the Lebesgue measure, one can find the basic properties of the Riesz capacity for example in~\cite{AH}, \cite{AE}, \cite{Me} and \cite{Mi}. In the metric case, assuming only the doubling property from the Borel regular measure $\mu$, we begin by showing that the Riesz $(\gamma,p)$-capacity is an outer measure. This means that the Riesz capacity satisfies the properties of Definition 2.3.

\begin{theorem}\label{Nov16}
The Riesz $(\gamma,p)$-capacity is an outer measure.
\begin{proof}
Clearly $\mathcal{C}_{\gamma,p} (\emptyset)=0$, since $0$ is an admissible function. The definition of the capacity also implies monotonicity, since if $E_{1}\subset E_{2}$, then $\mathcal{A}(E_2) \subset \mathcal{A}(E_1) $. 

To prove the countable subadditivity, let $\left\{A_i\right\}_{i=1}^{\infty}$ be a sequence of sets in $X$ and let $A=\bigcup_{i=1}^{\infty}A_i$. We may assume that $\sum_{i=1}^{\infty} \mathcal{C}_{\gamma,p}(A_i)<\infty$. Then, $\mathcal{C}_{\gamma,p}(A_i)<\infty$ for all $i \in \mathbb{N}$. Let $\epsilon>0$, and for each $i \in \mathbb{N}$, let $f_i \in \mathcal{A}(A_i)$ be such that $$||f_i||_{L^p(X)}^p<\mathcal{C}_{\gamma,p}(A_i) + \epsilon 2^{-i}.$$ We define $f(x):= \underset{i \in \mathbb{N}}{\sup} f_i(x)$. We have that $f(x)^p \leq \sum_{i=1}^{\infty}f_i(x)^p$, which implies $$||f||_{L^p(X)}^p \leq \sum_{i=1}^{\infty}||f_i||_{L^p(X)}^{p} \leq \sum_{i=1}^{\infty}\left(\mathcal{C}_{\gamma,p}(A_i)+\epsilon 2^{-i}\right) = \sum_{i=1}^{\infty}\mathcal{C}_{\gamma,p}(A_i)+\epsilon.$$ Moreover, we have that $I_\gamma f(x) \geq I_\gamma f_i(x)$, since $f(x)\geq f_i(x)$ for all $x \in X$ and $i \in \mathbb{N}$. Let $x \in A$. Then there exists $j \in \mathbb{N}$ such that $x \in A_j$ and hence $I_\gamma f(x) \geq I_\gamma f_j(x) \geq 1$. Thus $f$ is an admissible function for $A=\bigcup_{i=1}^{\infty}A_i$. Now $$\mathcal{C}_{\gamma,p}\Big( \bigcup_{i=1}^{\infty}A_i\Big) \leq ||f||_{L^p(X)}^p \leq \sum_{i=1}^{\infty}\mathcal{C}_{\gamma,p}(A_i)+\epsilon,$$ and the claim follows by letting $\epsilon \rightarrow 0$.
\end{proof}
\end{theorem}

\begin{remark}
The Riesz potential, as defined in \eqref{our definition}, is lower semicontinuous. For our purposes, it is enough to prove the lower semicontinuity for functions $f \in L^p_{+}(X)$. Let $x_0 \in X$. Then
$$I_{\gamma} f(x_0) = \int_{X} \frac{f(y)}{\mu\left(B(x_0,d(x_0,y))\right)^{1-\gamma}} \, d\mu (y) .$$
We need to show that 
$$ I_{\gamma} f(x_0)  \leq \liminf_{x \to x_0}  I_{\gamma} f(x),$$ 
when $x \to x_0$. Let $x \in X$. Since for any $y \in X$
$$B(x, d(x,y))  \subset B(x_0, d(x,y) + d(x_0, x)),$$
we have by the monotonicity of $\mu$ that 
$$\mu\left(B(x, d(x,y))\right) \leq \mu\left(B(x_0, d(x,y) + d(x_0, x))\right).$$
The above inequality and equality \eqref{spheres} imply that
%i.e.
%$$d(x_0, y) = \lim_{x \to x_0} d(x, y),$$ 
\begin{align*}
\limsup_{x \to x_0} \mu\left(B(x, d(x,y))\right)&\leq \lim_{x \to x_0}  \mu( B( x_0, d(x,y) + d(x_0, x))) \\
& =\mu\left(B(x_0, d(x_0, y))\right).
\end{align*}
Now, for any $y \in X$,
$$\frac{f(y)}{  \mu\left( B(x_0, d(x_0, y))\right)^{1-\gamma}} \leq \liminf_{x \to x_0 } \frac{f(y)}{\mu\left( B(x,d(x,y) ) \right)^{1-\gamma}}$$ and, by using Fatou's lemma, we get
\begin{align*}
I_{\gamma} f(x_0) &= \int_{X} \frac{f(y)}{\mu\left(B(x_0,d(x_0,y))\right)^{1-\gamma}} \, d\mu (y) \nonumber \\
&\leq \int_X \liminf_{x \to x_0 } \frac{f(y)}{\mu\left( B(x,d(x,y) ) \right)^{1-\gamma}} d \mu(y) \nonumber\\
&\leq \liminf_{x \to x_0 } \int_X  \frac{f(y)}{\mu\left( B(x,d(x,y) ) \right)^{1-\gamma}} d \mu(y) \nonumber\\
& = \liminf_{x \to x_0 } I_{\gamma}f(x).
\end{align*}

In addition, we get the following lower semicontinuity property of the Riesz potential as an operator
$$ I_{\gamma} f \leq \liminf_{i \to \oo} I_{\g}f_i,$$ when $f_i \rightarrow f$ weakly in $L^p_{+}(X)$. The weak convergence implies that $f_i \mu \rightarrow f \mu$ converge weakly as measures with the vague topology of \cite[Section 1.1]{Fu}. Also, because of equality \eqref{spheres} in the previous section, we have that our Riesz kernel is continuous and hence lower semicontinuous. The result then follows from \cite[Lemma 2.2.1.b)]{Fu} and \cite[Theorem 1.2. p.58]{Mi}.
%Indeed, let $f =\liminf_{i \to \oo} f_i$. Then, by Fatou's lemma, we have for all $x\in X$,
%\begin{align*}
%I_{\gamma} f(x) &= \int_{X} \frac{f(y)}{\mu\left(B(x,d(x,y))\right)^{1-\gamma}} \, d\mu (y)  = \int_{X} \frac{ \liminf_{i \to \oo} f_i(y)} %{\mu\left(B(x,d(x,y))\right)^{1-\gamma}} \, d\mu (y)  \\
%&\leq  \liminf_{i \to \oo} \int_{X} \frac{ f_i(y)}{\mu\left(B(x,d(x,y))\right)^{1-\gamma}} \, d\mu (y) = \liminf_{i \to \oo} I_{\g}(f_i)(x),
%\end{align*}
%and hence 
%$$ I_{\gamma} f \leq \liminf_{i \to \oo} I_{\g}f_i.$$
\end{remark}
Using the fact that the Riesz potential of a function $f$ is lower semicontinuous, we show that the Riesz capacity is an outer capacity. This means that the capacity of a set $E \subset X$ can be obtained by approximating with open sets from the outside.

\bt\label{t3}
$\mathcal{C}_{\gamma,p}$ is an outer capacity, that is, $$\mathcal{C}_{\gamma,p}(E)=\inf\left\{\mathcal{C}_{\gamma,p}(O): O \supset E, \, O \text{ open}\right\}.$$
\begin{proof}
By the monotonicity, $\mathcal{C}_{\gamma,p}(E)\leq \inf\left\{\mathcal{C}_{\gamma,p}(O): O \supset E, \, O \text{ open}\right\}$. To prove the inequality to the reverse direction, we may assume that $\mathcal{C}_{\gamma,p}(E)<\infty$. Let $0<\epsilon<1$ and let $f \in L^p_{+}(X)$ be a function such that $I_\gamma f \geq 1$ on $E$ and $$||f||_{L^p(X)}^p< \mathcal{C}_{\gamma,p}(E) + \epsilon.$$ We define $$f_\epsilon := \frac{1}{1-\epsilon} \, f$$ and $$G:= \left\{x \in X: I_\gamma f_\epsilon (x)>1\right\}.$$ Since $I_\gamma f_\epsilon$ is lower semicontinuous, $G$ is an open set. We also have that $f_\epsilon (x) > f (x)$ for all $x \in X$, since $0<\epsilon<1$. Now, if $x \in E$, then $I_\gamma f(x) \geq 1$ and hence $I_\gamma f_\epsilon (x) >1$. Thus we have that $x \in G$ and $E \subset G$. Moreover, $f_\epsilon$ is admissible for $\mathcal{C}_{\gamma,p}(G)$ and 
\begin{align*}
\mathcal{C}_{\gamma,p}(G) &\leq ||f_\epsilon||_{L^p(X)}^p= \Big(\frac{1}{1-\epsilon}\Big)^p ||f||_{L^p(X)}^p \\
&\leq \mathcal{C}_{\gamma,p}(E) (1-\epsilon)^{-p} + \epsilon (1-\epsilon)^{-p}.
\end{align*}
Since we have that $\inf\left\{\mathcal{C}_{\gamma,p}(O): O \supset E, \, O \text{ open}\right\}\leq \mathcal{C}_{\gamma,p}(G)$, letting $\epsilon \rightarrow 0$ yields the inequality to the other direction.
\end{proof}
\et

The next capacitary weak type lemma shows in particular that the Riesz potential $I_\gamma f$ of a nonnegative $L^p$-function $f$ belongs to the weak capacitary $L^p$-space.

\begin{lemma}
If $f \in L^p_{+}(X)$, then the capacitary weak type estimate $$\mathcal{C}_{\gamma,p}\left(\left\{x \in X : I_\gamma f(x) > a\right\}\right) \leq a^{-p} ||f||_{L^p(X)}^p$$ holds for each $0<a<\infty$. Moreover, $I_{\g} $ is bounded from $L^p_{+}(X)$ to $L^{p,\oo}(\mathcal{C}_{\g,p}).$
\end{lemma}
\begin{proof}
Let $f \in L^p_{+}(X)$ and $0<a<\infty$. We define $$f_a:= \frac{f}{a}$$ and $$F:= \left\{x \in X: I_\gamma f(x) > a\right\}.$$ Since $I_\gamma f_a = I_\gamma \big(\frac{f}{a}\big) \geq 1$ on $F$, $f_a$ is admissible for $F$ and $$\mathcal{C}_{\gamma,p}(F) \leq ||f_a||_{L^p(X)}^p= a^{-p}||f||_{L^p(X)}^{p}.$$
Moreover, the capacitary weak type estimate 
%$$a^{p} \, \mathcal{C}_{\gamma,p}\left(\left\{x \in X : I_\gamma f(x) > a\right\}\right) \leq  ||f||_{L^p(X)}^p, $$ 
%which is valid for any $0 <a < \oo$, 
implies that
\beqy
 \| I_{\g } f\|_{L^{p,\oo}(\mathcal{C}_{\g,p})} =\sup_{t >0} t \, \mathcal{C}_{\gamma,p}\left(\left\{x \in X : I_\gamma f (x) >  t \right\}\right)^{1/p} \leq  ||f||_{L^p(X)},\eeqy
and the second claim follows.
\end{proof}

We use the capacitary weak type estimate to prove the next theorem, which in particular says that the Riesz potential of a function $f \in L^p_{+}(X)$ is finite $\mathcal{C}_{\gamma,p}$- q.e. It follows that $I_{\gamma} f$, for $f \in L^p(X)$, is well-defined $\mathcal{C}_{\gamma,p}$- q.e. and that the Riesz potential, as an operator, is linear outside a set of capacity zero.
%and that the Riesz potential, as an operator, is linear outside a set of capacity zero.
%which in particular says that the Riesz potential of a function $f \in L^p_{+}(X)$ is finite $\mathcal{C}_{\gamma,p}$- q.e.
\bt\label{t4}
Let $E \subset X$. Then $\mathcal{C}_{\gamma,p}(E)=0$ if and only if there exists $f \in L^p_{+}(X)$ such that $I_\gamma f(x)=\infty$ for all $x \in E$.
\begin{proof}
If $\mathcal{C}_{\gamma,p}(E)=0$, then for any integer $j$, we can find an admissible function $f_j \in \mathcal{A}(E)$ such that $$||f_j||_{L^p(X)}^p= \int_{X} f_j(y)^p d\mu(y) < 2^{-j}.$$ Then the function $f:= \sum_{j=1}^{\infty} f_j$ belongs to $L^p(X)$ and $$I_\gamma f(x)= \int_{X} \frac{\sum_{j=1}^{\infty}f_j(y)}{\mu\left(B(x,d(x,y))\right)^{1-\gamma}} \, d\mu (y) = \sum_{j=1}^{\infty} I_\gamma f_j(x)= \infty$$ for all $x \in E$, since $I_\gamma f_j \geq 1$ on $E$ for each $j$. 

Conversely, if there exists a nonnegative function $f \in L^p(X)$ such that $I_\gamma f= \infty$ on $E$, then by the capacitary weak type estimate $$\mathcal{C}_{\gamma,p}(E) \leq \mathcal{C}_{\gamma,p}\left(\left\{x \in X: I_\gamma f(x) > a\right\}\right) \leq a^{-p} ||f||^p_{L^p(X)}$$ for every $a>0$. By letting $a \rightarrow \infty$, we see that $\mathcal{C}_{\gamma,p}(E)=0$.
\end{proof}
\et

\begin{corollary}\label{t5} 
Let $f_1$, $f_2$, $f \in L^p(X)$. Then 
$$I_\g (f_1+f_2) = I_\g (f_1)+ I_\g(f_2), \qquad \mathcal{C}_{\g,p}\text{-q.e.}$$
and
$$I_\g (a f) = a I_\g(f), \qquad \mathcal{C}_{\g, p}\text{-q.e.},$$
where $a$ is any finite constant.
\begin{proof} 
If each term on the right side of the above equalities is finite at a point $x \in X$, then the equalities hold at such point by the definition of the Riesz potential. By Theorem~\ref{t4}, the sets where the equalities can fail are of capacity zero. 
%and hence they hold $\mathcal{C}_{\g, p}$-q.e.
\end{proof}
\end{corollary}

Next, we are going to prove several convergence results. We start by defining the convergence of a sequence of functions in capacity.

\begin{definition}\label{int16} We say that a sequence $\{ f_i\}$ converges in capacity to $f$, denoted $f_i \rightarrow f$ in $\mathcal{C}_{\g, p}$, if for every $\epsilon >0$ $$\lim_{i \rightarrow \infty} \mathcal{C}_{\g, p}\left(\{x \in X: |f_i(x) - f(x)| > \epsilon \}\right)= 0.$$ \end{definition}

We show that the $L^p$-convergence of functions implies that the corresponding sequence of the Riesz potentials converges in capacity. Also, for a subsequence, we have pointwise convergence except for a set of capacity zero.

\begin{theorem}\label{t7} Let $\left\{f_i\right\} \subset L^p_{+}(X)$ and $f \in L^p_{+}(X)$. Each of following statements is a consequence of the previous one. 
\bit
\item[(i)] $f_i \rightarrow f$ in $L^p(X)$
\item[(ii)] $I_{\g} f_i \rightarrow I_{\g} f$ in $\mathcal{C}_{\g, p}$
\item[(iii)] There exists a subsequence $\{f_{i_j}\}$ of $\{f_i\}$ such that 
$$I_{\g} f_{i_j} \rightarrow I_{\g} f \, \text{ pointwise } \, \mathcal{C}_{\g,p}\text{-q.e.}$$
%\item[(iv)]$ I_{\g} f_{i_j} \rightarrow I_{\g}f \, \, \mathcal{C}_{\g,p}\text{-q.e.}$
\eit
\begin{proof} We show first that (i) implies (ii). Let $\epsilon >0$. By Theorem 3.6, the potentials $I_{\g} f_i $ and $ I_{\g} f$ are finite $\mathcal{C}_{\g, p}$-q.e. Then, we have by Corollary~\ref{t5} and by Lemma 3.5 that 
$$\mathcal{C}_{\g,p}\left(\left\{x \in X : | I_\gamma f_i(x)  - I_{\g} f (x) | > \e \right\}\right) \leq \e^{-p} ||f_i - f ||_{L^p(X)}^p,$$
which proves the claim. 

Next, we assume that (ii) holds and show that it implies (iii). Let $\e=2^{-j}$. Then, there exists a subsequence $\{f_{i_j}\}$ such that $$\mathcal{C}_{\g,p}\left(\left\{x \in X : | I_\gamma (f_{i_j})(x)  - I_{\g} f (x) | > 2^{-j} \right\}\right) < 2^{-j}.$$ We use the notation $A_j=\left\{x \in X : | I_\gamma (f_{i_j})(x)  - I_{\g} f (x) | > 2^{-j} \right\}$. The upper limit set $$A=\bigcap_{k=1}^{\infty}\bigcup_{j=k}^{\infty} A_j$$ has zero capacity, since $$\mathcal{C}_{\gamma,p}(A) \leq \sum_{j=k}^{\infty}\mathcal{C}_{\gamma,p}(A_j) \leq \sum_{j=k}^{\infty} 2^{-j}$$ for all $k$. Now, if $x \in X \setminus A$ then there exists $k=k(x)$ such that $x \in X \setminus A_j$ for $j \geq k$, that is $$| I_\gamma (f_{i_j})(x)  - I_{\g} f (x) | \leq 2^{-j}.$$ This implies that $I_\gamma (f_{i_j})(x) \rightarrow I_{\g} f (x)$, as $j \rightarrow \infty$, which proves the claim.

%%%%%%%%%%%%%%%%%%%%%%
%There exists a subsequence $\{f_{i_{j}}\}$ and sets ${A_j}$, where $\mathcal{C}_{\g,p}(A_j) \leq \e \, 2^{-j}$, for which
%$$ | I_{\g}(f_{i_{j}} ) (x) - I_{\g} f(x) | \leq \frac{1}{j} \qquad \text{ except on } A_j.$$
%Therefore 
%$I_{\g}(f_{i_{j}} )  \rightarrow I_{\g}(f ), $ as $j\to \oo$, uniformly in $ X \setminus \cup_{j} A_j$, where
%$\mathcal{C}_{\g,p}(\cup_{j} A_j) \leq \e$. Then, a simple diagonalization argument proves the claim. 
%%Finally, (iv) obviously follows from (iii).

\end{proof}
\end{theorem}

By using the above theorem, we can strengthen the lower semicontinuity property of $I_\gamma$ from Remark 3.3, at least outside a set of capacity zero.

\begin{theorem}\label{t6} Let $\left\{f_i\right\} \subset L^p(X)$ and $f \in L^p(X)$.
\bit \item[(i)] If $f_i \rightarrow f$ weakly in $L^p(X)$, then 
$$ \liminf_{i \to \oo} I_{\g} f_i \leq I_{\g} f \leq \limsup_{i \to \oo} I_{\g} f_i, \qquad \mathcal{C}_{\g, p}\text{-q.e.}$$
\item[(ii)] If $f_i \rightarrow f$ weakly in $L^p_{+}(X)$, then 
$$ I_{\g} f \leq \liminf_{i \to \oo} I_{\g} f_i \, \text{ everywhere }$$
and
$$ I_{\g} f = \liminf_{i \to \oo} I_{\g} f_i \qquad \mathcal{C}_{\g, p}\text{-q.e.}$$
\eit
\begin{proof}  
We prove first the claim (i). By the Banach-Saks Theorem (see \cite{BaSa}), there exists a subsequence $\{f_i'\}$ such that a sequence $\{g_j\}$, where
$$ g_j = j^{-1} \sum_{i=1}^{j} f_i',$$ converges to $f$ in $L^p(X)$. Then, by Theorem~\ref{t7}, there exists a subsequence $\{g_j'\}$ such that $$ I_{\g} f = \lim_{j \to \oo} I_{\g} g_j'\qquad \mathcal{C}_{\g,p}\text{-q.e.}$$
Then, the left inequality in (i) follows due to the fact that
$$ I_{\g} f= \lim_{j \to \infty} I_{\g} g_j' \geq \liminf_{i \to \infty} I_{\g} f_i' \geq \liminf_{i \to \infty} I_{\g} f_i\, \qquad \mathcal{C}_{\g, p}\text{-q.e.}$$
The right inequality in (i) follows by replacing $f_i $ and $f$ by $-f_i$ and $-f$ in the previous argument. 

If $f_i \rightarrow f$ weakly in $L^p_{+}(X)$, then by the lower semicontinuity of $I_{\g} (\cdot)$ (see Remark 3.3), we have that 
$$ I_{\g} f \leq \liminf_{i \to \oo} I_{\g}f_i \, \, \text{ everywhere}$$
and it follows by (i) that
$$ I_{\g} f =  \liminf_{i \to \oo} I_{\g}f_i \qquad \mathcal{C}_{\g, p}\text{-q.e.}$$
\end{proof}
\end{theorem}

We prove two more convergence results for the Riesz capacity. As a corollary of Theorem 3.12, we get a lower semicontinuity property for the capacity that is an analogue of Fatou's lemma.

\begin{theorem}\label{t1}
If $X \supset K_1 \supset K_2 \cdots$ are compact sets and $K= \bigcap_{i=1}^{\infty} K_j$, then $$\lim_{i\rightarrow \infty} \mathcal{C}_{\gamma,p} (K_i)= \mathcal{C}_{\gamma,p}(K).$$
\begin{proof}
Clearly, by the monotonicity, $\lim_{j\rightarrow \infty} \mathcal{C}_{\gamma,p} (K_j) \geq \mathcal{C}_{\gamma,p}(K)$. On the other hand, let $O$ be an open set containing $K$. By the compactness of $K$, we have that $K_j \subset O$ for all sufficiently large $j$. Then $\lim_{j\rightarrow \infty} \mathcal{C}_{\gamma,p} (K_j) \leq \mathcal{C}_{\gamma,p}(O)$. 
Finally, since $\mathcal{C}_{\gamma,p}$ is an outer capacity by Theorem~\ref{t3},
$$\lim_{j\rightarrow \infty} \mathcal{C}_{\gamma,p} (K_j) \leq \inf\left\{\mathcal{C}_{\gamma,p}(O): O \supset K, \, O \text{ open}\right\} = \mathcal{C}_{\g, p}(K).$$
\end{proof}
\end{theorem}

\begin{theorem}\label{t2}
If $A_1 \subset A_2 \subset \cdots$ are subsets of $X$ and $A=\bigcup_{i=1}^\infty A_i$, then 
$$\lim_{i\rightarrow \infty} \mathcal{C}_{\gamma,p} (A_i)= \mathcal{C}_{\gamma,p}(A),$$
that is, the Riesz capacity $\mathcal{C}_{\g,p}$ is a Fatou capacity.
\begin{proof}
We may assume that $\lim_{i\rightarrow \infty} \mathcal{C}_{\g,p}(A_i)=l<\infty$. Let $f_i \geq 0$ be a test function for 
$\mathcal{C}_{\g,p}(A_i)$ such that 
\beq\label{eq1} \|f_i\|_{L^p(X)}^p \leq \mathcal{C}_{\g,p}(A_i ) + \frac{1}{i}.\eeq
Then, the sequence $\{f_i\}$ is bounded in $L^p(X)$ and there exists a subsequence $\{f_{i_j}\}$ that converges weakly to a function $f \in L^p_{+}(X)$. By Theorem~\ref{t6} (ii), we have that 
$$ I_{\g} f \geq 1\, \text{ on } A_i\, \qquad \mathcal{C}_{\g,p}\text{-q.e.}$$
and hence 
$$I_{\g} f \geq 1\, \text{ on } A\, \qquad \mathcal{C}_{\g,p} \text{-q.e.}$$
Let $E$ be the subset of $A$, where the previous inequality holds. Then, by (\ref{eq1}) and the weak convergence of the functions 
$$ \mathcal{C}_{\g,p}(A) = \mathcal{C}_{\g,p}(E) \leq \| f\|_{L^p(X)}^p \leq \liminf_{j \rightarrow \infty} ||f_{i_j}||_{L^p(X)}^p \leq l,$$
from which the result follows. Here, we also used the the lower semicontinuity of $|| \cdot ||_p^p$ (see \cite[Lemma 3.1.2. p.109]{AE}).
\end{proof}
\end{theorem}

\begin{corollary}
If $\{A_i\}_{i=1}^{\oo}$ is a sequence of sets in $X$, then 
$$ \mathcal{C}_{\g,p}(\liminf_{i \to \oo} A_i) \leq \liminf_{i \to \oo} \mathcal{C}_{\g,p}(A_i).$$
\begin{proof}
Let $S:= \liminf_{i \to \infty} A_i = \bigcup_{j} \bigcap_{k \geq j} A_k$ and $S_i :=\bigcup_{j=1}^i \bigcap_{k \geq j} A_k. $ Then $S_i \subset S_{i+1} \subset \cdots$ and $S=\bigcup_{i=1}^\infty S_i$, and by Theorem~\ref{t2}, 
$$ \mathcal{C}_{\g,p}(S) = \lim_{i \to \oo} \mathcal{C}_{\g,p}(S_i)  \leq \liminf_{i \to \oo} \mathcal{C}_{\g,p}(A_i).$$
\end{proof}
\end{corollary}

%%%%%%%%%%%%%%%%%%%%%%%%%%%%%%%%%%%%%%%%%%%%%%%%%%%%%%%%%%%%%%%%
%The distribution function $C_{\g,p}^f$ and the nonincreasing rearrangement
%$f^\star_{C_{\g,p}}$ are defined as in the case of measures by
%$$
%C_{\g,p}^f(t):=C_{\g,p}\{|f|>t\},$$
%and
%$$
%f^\star_{C_{\g,p}}(x):=\inf\{t;\, C_{\g,p}\{ |f|>t\}\leq x\}=\int_0^\oo \chi_{[0,C_{\g,p}\{ |f|>t\})}(x)\, dt,
%$$
%since $\{t;\, C_{\g,p}\{ |f|>t\}\leq x\}$ is the interval $[f^\star_{C_{\g,p}}(x),\oo]$ .
%
%Since the capacity $C_{\g,p}$ is a Fatou capacity, we obtain the following properties by~\cite[Theorem 1]{CMS}: 
%\item[(a)] If $|f|\leq \liminf_{i \to \oo} |f_i| $, then $$ f^\star_{C_{\g,p}}\leq \liminf_{i \to \oo} (f_i)^\star_{C_{\g,p}}.$$
%\item[(b)] Fatou's lemma for $C_{\g,p}$,
%$$\int(\liminf_{i \to \oo}|f_i|)\, d C_{\g,p}\leq\liminf_{i \to \oo} \int |f_i|\, d C_{\g,p}.$$
%\item[(c)] If $0\leq f_i \uparrow f$, then $$(f_i)^\star_{C_{\g,p}}\uparrow f^\star_{C_{\g,p}}.$$
%\eit
%%%%%%%%%%%%%%%%%%%%%%%%%%%%%%%%%%%%%%%%%%%%%%%%%%%%%%%%%%%%%%%%%%%
The next definition extends the outer capacity property of Theorem 3.4 to the case where the Riesz capacity of a set $E \subset X$ can also be obtained by approximating with compact sets from the inside. By Theorem 3.15, we have this inner capacity property for the Riesz capacity, when considering analytic sets (for the definition of analytic sets, we refer to e.g. \cite{AE}, \cite{Ch}, \cite{M85}).

\begin{definition}
A set $E \subset X$ is called $\mathcal{C}_{\g, p}$-capacitable, if
\begin{align*}
\mathcal{C}_{\g, p}(E)&=\sup\{ \mathcal{C}_{\g, p}(K) : K \subset E, \, K \, \text{compact} \} \\
&= \inf\{  \mathcal{C}_{\g, p}(O) :  O \supset E, \, O \, \text{open} \}.
\end{align*}
\end{definition}

Capacitability has been studied in a very general context by Choquet in~\cite{Ch}. Other references are \cite{A}, \cite{AH}, \cite{AE}, \cite{Fe}, \cite{M85} and \cite{Me}, and the references therein. Choquet's capacitability theorem (see \cite[pp.182--184]{AE}) together with Theorems~\ref{t1} and~\ref{t2} give the next theorem, which says that all analytic sets are $\mathcal{C}_{\g, p}$-capacitable. In particular, we have that all Borel sets are $\mathcal{C}_{\g, p}$-capacitable, which means that the Riesz capacity is a so called Choquet capacity.

\begin{theorem}\label{An}
All analytic sets, and hence all Borel sets, are $\mathcal{C}_{\g, p}$-capacitable.
\end{theorem}

%We study capacitary distributions and potentials for $\mathcal{C}_{\g,p}.$ 
%Let $A$ be any set with $\mathcal{C}_{\g,p}(A) < \oo$. 
For the following variational problem 
%and try to find an admissible function $f$ such that the following minimum is obtained:
\beq\label{p1}
\min\Big\{ \|f\|_{L^p(X)}^p: f \in L^p(X), I_{\g} f \geq 1 \, \, \mathcal{C}_{\g,p}\text{-q.e.} \text{ on } A\Big\}, \eeq
we call a solution $f$ a $\mathcal{C}_{\g,p}$-capacitary distribution of $A$ and $I_{\g}f$ a $\mathcal{C}_{\g,p}$-capacitary potential of $A$. 

\begin{theorem}
 If $\mathcal{C}_{\g,p}(A) < \oo$, then $A$ has a unique $\mathcal{C}_{\g,p}$-capacitary distribution $f$ for which $f \in L^p_{+}(X)$, $\|f\|_{L^p(X)}^p = \mathcal{C}_{\g,p}(A)$ and 
$$ \int_{X} f(x)^{p-1} g(x) d \mu(x) \geq 0$$ 
for all $g \in L^p(X)$ such that 
$$ I_{\g} g \geq 0 \,\, \, \mathcal{C}_{\g,p}\text{-q.e.} \text{ on } A.$$
\begin{proof} 
Using Clarkson's inequality for $L^p$-norms, the theorem follows as in~\cite[Theorem 9]{Me} due to previous results in the paper.
\end{proof}
\end{theorem}

\subsection*{Dual Riesz capacity} 
Let $\nu$ be a positive measure on $X$, $1 < p < \oo$ and $A \in \mathcal{F}$, where $\mathcal{F}$ is the $\s$-algebra of sets which are $\nu$-measurable for all positive measures $\nu$ with finite total variation in $X$. The total variation of any such $\nu$ in $X$ is
$$\| \nu\| = \nu(X)= \sup\Big\{ \nu(A): A \subset X, \, A \text{ is measurable}\Big\}.$$
%\\
%&+& \inf\Big\{ \nu(A): A \subset X, A \text{ is measurable }\Big\}\\
%& =& \sup\Big\{ \nu(A): A \subset X, A \text{ is measurable }\Big\}.
In the case of a measure of this type, we define 
$$I_{\g} \nu(x)= \int_{X} \frac{ d \nu(y)}{ \mu( B(x, d(x,y)))^{1- \g}},$$
and following~\cite{Me} we introduce a capacity, which uses measures as test elements
\begin{align*}
 c_{\g,p}(A)=\sup\Big\{ &\| \nu\| : \nu \text{ is a positive measure in $X$}, \, \, \nu(X \setminus A)=0, \\
&||\nu||<\infty \text{ and } \| I_{\g} \nu\|_{L^{p'}(X)} \leq 1 \Big\},
\end{align*}
where 
%$$\| I_{\g} \nu\|_{L^{p'}(X)}= \Big( \int_{X}\Big| \int_{X} \frac{d \nu(y)}{ \mu( B(x, d(x,y)))^{1 - \g}} \Big|^{p'} d \mu(x) \Big)^{1/p'}$$
$\frac{1}{p}+\frac{1}{p'}=1$. 
%From the definition we see that for $A \in \mathcal{F}_1$
%\beq\label{S26}
%c_{\g,p}(A) = \Big( \inf_{\nu} \| I_{\g} \nu\|_{L^{p'}(X)} \Big)^{-1},\eeq where $\nu$ is a positive measure supported on $A$ such that $\| \nu \| \geq 1 %$. We may also restrict $\| \nu \| $ to one.  

With the same techniques as in \cite[Theorem 12, Theorem 14]{Me} or \cite[pp. 114--117]{AE}, we see that $c_{\g,p}$ is an inner capacity on $\mathcal{F}$ that satisfies
\begin{equation}
\label{dual}
c_{\g,p}(A) = \mathcal{C}_{\g,p}(A)^{1/p}.
\end{equation}
%and that all analytic sets are $c_{\g,p}$- capacitable on $\mathcal{F}$. 
Indeed, for the inequality $c_{\g,p}(A) \leq \mathcal{C}_{\g,p}(A)^{1/p}$, let $f \in \mathcal{A}(A)$ and let $\nu$ be a test measure for $c_{\g,p}(A)$. Then, by H\"older's inequality,
\begin{align*}
&\nu(A)\leq\int_{A} I_{\g} f (x) d \nu(x)  = \int_{X} \int_{A} \frac{1}{ \mu( B(x, d(x,y)) )^{1-\g}  } d \nu(x) f(y)  d \mu(y) \\
&\leq \Big( \int_{X}  \Big( \int_{A} \frac{1}{\mu(B(x, d(x,y)))^{1-\g} } d \nu(x)   \Big)^{p'}  d \mu(y)  \Big)^{\frac{1}{p'}} \Big(\int_{X} f(y)^p  d \mu(y) \Big)^{\frac{1}{p}} \\
&\leq \|f \|_{L^p(X)} \| I_{\g} \nu\|_{L^{p'}(X)} \leq \|f \|_{L^p(X)}
\end{align*}
and the inequality
$$ c_{\g,p}(A) \leq \mathcal{C}_{\g,p}(A)^{1/p}$$
follows by taking the infimum over admissible functions $f$ and the supremum over the test measures $\nu$.
In order to obtain the equality \eqref{dual}, one can show that (see \cite{AE} or \cite{Me})
$$ \mathcal{C}_{\g,p}(K)^{-1/p}= \sup_{f}\big\{  \inf_{x \in K} I_{\g} f(x) : f \in L^p_{+}(X), \| f\|_{L^p(X)} \leq 1 \big\}$$
and 
$$c_{\g,p}(K) = \Big( \inf_{\| \nu \| = 1} \| I_{\g} \nu\|_{L^{p'}(X)} \Big)^{-1} = \Big( \inf_{\| \nu \| = 1} \sup_{f \in L^p_{+}(X)} | \int_{X} ( I_{\g} \nu) f d\mu | \Big)^{-1}$$
for compact sets $K$, where $\nu$ is a positive measure supported on $K$ and $\|f\|_{L^p(X)} \leq 1.$
Thus, as in \cite[Theorem 14]{Me} or \cite[Theorem 3.6.1, p.115]{AE}, the equality \eqref{dual} follows with the use of the Minimax Theorem (see~\cite{F}) and for general sets $A$ by a capacitability argument.
%Theorem~\ref{An}.
%and also the capacitability of $c_{\g,p}$, equality \eqref{dual} follows. 

\section{Capacity estimates and Hausdorff measure}

In this section, we define a Hausdorff measure based on the upper bound estimate for the capacity of balls. We prove an upper bound estimate for the capacity of a set $E \subset X$ in terms of this modified Hausdorff content. We also show that the Hausdorff content, satisfying a condition placed by $\gamma p$, is zero if the capacity of the set is zero. For the latter result, we assume that our space $X$ satisfies inequality \eqref{tuplaavuus toinen suunta} which holds in connected spaces. Similar results for compact sets can be found in \cite[Theorem 2.2]{Sj}, where the assumption of connectedness is replaced by a density condition that gives an inequality equivalent to \eqref{tuplaavuus toinen suunta}. In this section, we give direct and short proofs to results that apply not only for compact sets. In particular, we do not need to use a version of Frostman's lemma in our proofs. Also, unlike in \cite{Sj}, we do not assume our space to be complete. We start by showing that the Riesz capacity of a ball is bounded from above by a constant times the measure of the ball to the power $1-\gamma p$.

\begin{lemma}
Let $1<p<\infty$ and $0<\gamma<1$ be such that $\gamma p <1$. Then $$\mathcal{C}_{\gamma, p} \left(B(x,r)\right) \leq C \, \mu \left(B(x,r)\right)^{1-\gamma p}.$$ 
%where $C=c_d^{2p} \, 3^{Q \left(1-\gamma\right)p}$, $c_d$ is the doubling constant in \eqref{doubling measure} and $Q = \log_{2} c_d$.
\begin{proof}
%We calculate an upper bound for the $(\gamma,p)$- capacity of a ball, when $\gamma p <1$. 
Choose $$g= c^2 \, 3^{Q \left(1-\gamma\right)} \frac{\chi_{B(x,r)}}{\mu\left(B(x,r)\right)^{\gamma}},$$ where $c>0$ and $Q>0$ are some constants, for which the inequality \eqref{doubling dimension} holds. For each $z \in B(x,r)$, we have

\begin{align*}
I_{\gamma} g(z) 
%&= \int_{X} \frac{1}{\mu\left(B(z,d(z,y))\right)^{1-\gamma}} \cdot c^2 \, 3^{Q \left(1-\gamma\right)} \frac{\chi_{B(x,r)} (y)}{\mu\left(B(x,r)\right)^{\gamma}} \, d\mu (y) \\
&= \frac{c^2 \, 3^{Q \left(1-\gamma\right)}}{\mu\left(B(x,r)\right)^{\gamma}} \int_{B(x,r)} \frac{1}{\mu\left(B(z,d(z,y))\right)^{1-\gamma}} \, d\mu (y).
\end{align*}
For each $y \in B(x,r)$, we have that $d(z,y) \leq 2r$, since $z \in B(x,r)$. Now
$$\mu \left(B(z,d(z,y))\right) \leq \mu \left(B(z,2r)\right) \leq \mu \left(B(x,3r)\right),$$  for each $y \in B(x,r)$.
Then
\begin{align*}
I_{\gamma} g(z) &\geq \frac{c^2 \, 3^{Q \left(1-\gamma \right)}}{\mu\left(B(x,r)\right)^{\gamma}} \int_{B(x,r)} \frac{1}{\mu\left(B(x,3r)\right)^{1-\gamma}} \, d\mu (y) \\
&= c^2 \, 3^{Q \left(1-\gamma\right)} \cdot \frac{\mu\left(B(x,r)\right)^{1-\gamma}}{\mu\left(B(x,3r)\right)^{1-\gamma}} \\
&\geq c^2 \, 3^{Q \left(1-\gamma\right)} \cdot \frac{1}{c^2} \cdot \left(\frac{r}{3r}\right)^{Q \left(1-\gamma\right)} = 1 \, ,
\end{align*}
where the last inequality follows by \eqref{doubling dimension}. Thus $g$ is admissible and we get the upper bound

\begin{align*}
\mathcal{C}_{\gamma, p} \left(B(x,r)\right) &\leq ||g||_{L^p(X)}^p \\
&= c^{2p} \, 3^{Q \left(1-\gamma\right)p} \, \frac{\mu \left(B(x,r)\right)}{\mu \left(B(x,r)\right)^{\gamma p}} \\
&= C \, \mu \left(B(x,r)\right)^{1-\gamma p}.
\end{align*}
\end{proof}
\end{lemma}

%%%\noindent In $\mathbb{R}^n$, with $\alpha = n \gamma$, this corresponds to
%%%$$C_{\alpha, p} \left(B(x,r)\right) \leq C \, r^{n-\alpha p},$$ when $\alpha p <n$

Lemma $4.1$ leads us to define a modified version of the Hausdorff measure that works in our generality. In this section, (and throughout the paper) we do not assume the doubling measure $\mu$ to satisfy the regularity \eqref{ahlfors} or any other estimates that would give uniform lower bounds or upper bounds for the measure of balls in terms of the radii. Recall that the usual definition for the $\lambda$-Hausdorff content of a set $E \subset X$, for $0<r\leq \infty$, is
$$\mathcal{H}_r^\lambda(E)=\inf\big\{\sum_{i=1}^{\infty} r_i^\lambda: E \subset \bigcup_{i=1}^{\infty}B(x_i,r_i), \, x_i \in E, \, r_i \leq r\big\},$$
%If $h: [0, \infty) \rightarrow [0, \infty)$ is a non-decreasing function satisfying the doubling condition and $$h(0)=0,$$ then we say that $h$ is a %measure function. For $0<\delta \leq \infty$ and $E \subset X$ we define the $h$-Hausdorff content as %$$\mathcal{H}_\delta^h=\inf\big\{\sum_{i=1}^{\infty} h(r_i): E \subset \bigcup_{i=1}^{\infty}B(x_i,r_i), r_i \leq \delta\big\},$$ 
%and if $\delta =\oo$, then we obtain the Hausdorff capacity which is a concave capacity.  
and the $\lambda$-Hausdorff measure of $E$ is $\mathcal{H}^\lambda(E)= \lim_{r \rightarrow 0} \mathcal{H}_r^\lambda(E)$. The Hausdorff dimension of $E$ is the number $$\dim(E)= \inf\left\{ \lambda>0: \mathcal{H}^\lambda(E)=0\right\}.$$ 
Let $1<p<\infty$, $0<\gamma<1$ and $\gamma p <1$. In our case, we define the Hausdorff content of a set $E \subset X$, for $0<r\leq\infty$, as 
$$\widetilde{\mathcal{H}}_r^{\gamma,p}(E)= \inf\big\{\sum_{i=1}^{\infty} \mu(B(x_i,r_i))^{1-\gamma p}: E \subset \bigcup_{i=1}^{\infty}B(x_i,r_i), \, x_i \in E, \, r_i \leq r \big\}.$$
Then, the Hausdorff measure is $$\widetilde{\mathcal{H}}^{\gamma,p}(E)= \lim_{r \rightarrow 0} \widetilde{\mathcal{H}}_r^{\gamma,p}(E).$$
%=: \mathcal{H}_{\gamma,p}(E).
Note that if the measure is $Q$-regular, then $$\widetilde{\mathcal{H}}_r^{\gamma,p} \approx \mathcal{H}_r^{Q(1-\gamma p)}.$$
In the next theorem, we show that the Riesz capacity of a set $E \subset X$ is bounded from above by a constant times the (modified) Hausdorff content of the set $E$. In particular, this implies that compact sets with positive capacity have positive Hausdorff measure (see \cite[Theorem 2.2]{Sj}).

\begin{theorem}
Let $1<p<\infty$ and $0<\gamma<1$ be such that $\gamma p<1$. Then $\mathcal{C}_{\gamma,p}(E) \leq C \widetilde{\mathcal{H}}^{\gamma,p}_{\infty}(E)$ for each $E \subset X$, where $C$ is the same constant as in Lemma 4.1.
\begin{proof}
Suppose that $\widetilde{\mathcal{H}}^{\gamma,p}_{\infty}(E)<\infty$, otherwise the claim is obvious. For $\epsilon >0$, there is a countable covering $\left\{B(x_i,r_i)\right\}$ of $E$ such that 
%$r_i < \epsilon$ for each $i$ and 
$$\sum_{i=1}^{\infty} \mu(B(x_i,r_i))^{1-\gamma p} < \widetilde{\mathcal{H}}^{\gamma,p}_{\infty}(E) + \epsilon.$$ Now, by the monotonicity and Theorem 4.1
\begin{align*}
\mathcal{C}_{\gamma,p}(E) &\leq \sum_{i=1}^{\infty} \mathcal{C}_{\gamma,p} (B(x_i,r_i)) \\
&\leq C \sum_{i=1}^{\infty} \mu(B(x_i,r_i))^{1-\gamma p} \\
&< C (\widetilde{\mathcal{H}}^{\gamma,p}_{\infty}(E) + \epsilon).
\end{align*}
The claim follows by letting $\epsilon \rightarrow 0$.
\end{proof}
\end{theorem}

For the proof of the next theorem, we need an opposite inequality to \eqref{doubling dimension} which is true in connected spaces. Indeed, if $X$ is connected then by \cite[Corollary 3.8]{BB} there exist constants $C>0$ and $s>0$ such that for all balls $B(y,R)$ in $X$, all $z \in B(y,R)$ and all $0<r\leq R$, 
\begin{equation}\label{tuplaavuus toinen suunta}
\frac{\mu(B(z,r))}{\mu(B(y,R))} \leq C \left(\frac{r}{R}\right)^s.
\end{equation}
Note that inequality \eqref{tuplaavuus toinen suunta} given by the connectedness (or uniform perfectness) of the space $X$ is equivalent to the density condition assumed in \cite{Sj}. 
%Our next theorem gives as a special case the other part of the the main result of \cite{Sj} (see \cite[Theorem 2.2]{Sj}). 
The proof of the next theorem is direct and we do not need to use Frostman's lemma to obtain the result. In the Euclidean space, with the Lebesgue measure, we use the notation $\alpha = \gamma n$ and consider the usual Riesz potential of order $\alpha$ from Remark 2.2 (ii). Our result implies the classical result that if the Riesz capacity of a set $E$ is zero, then $E$ has Hausdorff dimension at most $n-\alpha p$, where $\alpha p <n$ (see e.g. \cite[Section 5.2, Theorem 2.3]{Mi}).  

\begin{theorem}
Assume that $X$ satisfies \eqref{tuplaavuus toinen suunta}. Let $1<p<\infty$, $1<\tilde{p}<\infty$, $0<\gamma <1$ and $0<\tilde{\gamma}<1$ be such that $\gamma p<1$, $\tilde{\gamma} \tilde{p}<1$ and $\tilde{\gamma} \tilde{p} < \gamma p$. 
If $\mathcal{C}_{\gamma,p}(E)=0$, then $\widetilde{\mathcal{H}}_{\infty}^{\tilde{\gamma},\tilde{p}}(E)=0$.

%%%%%%%Then $\widetilde{\mathcal{H}}_{\infty}^{\tilde{\gamma},\tilde{p}}(E) \leq C \mathcal{C}_{\gamma,p}(E)$.
\begin{proof}
In the following, we prove the result for bounded sets $E \subset X$. If the set $E$ is not bounded, then there exists bounded sets $E_j$ such that $E=\bigcup_{j=1}^{\infty} E_j$. We then use the countable subadditivity of the Hausdorff content, the monotonicity of the Riez capacity and the result for bounded sets to obtain the result for unbounded sets $E$.

%Indeed, by Theorem 3.15, the Riesz capacity of an analytic set can be obtained by approximating with compact sets from the inside. With the same techniques as in %Section 3, we can see that the same is true also for the Hausdorff content of the set $E$.

Let $E \subset X$ be a bounded set. For $\epsilon > 0$, there is an admissible function $f \geq 0$ such that $$||f||_{L^p(X)}^p < \epsilon.$$ For such a function $f$, at each point $x \in E$,
$$1 \leq \int_{X} \frac{f(y)}{\mu\left(B(x,d(x,y))\right)^{1-\gamma}} \, d\mu (y).$$ Let $x_0 \in E$. We choose $R_0>\diam(E)$ large enough such that $E \subset B(x_0, R_0)$ and that the integral below is more than one half. Notice that we can always find such a radius $R_0$ but the selection depends on the set $E$.
%Then $B(x_0,R_0) \subset B(x, R_0 + \diam(E))$ for each $x \in E$.  
%Define $R=R_0+\diam(E)=2 \diam(E)$ and 
Define $R=2R_0$ and $r_i=2^{-i}R$, for $i \in \mathbb{N}$. For each point $x \in E$
\begin{align*}
\frac{1}{2} &\leq \int_{B(x_0,R_0)} \frac{f(y)}{\mu\left(B(x,d(x,y))\right)^{1-\gamma}} \, d\mu (y) \\
&\leq \int_{B(x,R)} \frac{f(y)}{\mu\left(B(x,d(x,y))\right)^{1-\gamma}} \, d\mu (y)
\end{align*}
and hence
\begin{align*}
 1&\leq 2 \int_{B(x,R)} \frac{f(y)}{\mu\left(B(x,d(x,y))\right)^{1-\gamma}} \, d\mu (y) \\
&=2 \sum_{i=0}^{\infty} \int_{B(x,r_i)\setminus B(x,r_{i+1})} \frac{f(y)}{\mu\left(B(x,d(x,y))\right)^{1-\gamma}} \, d\mu (y) \\
&\leq 2 \sum_{i=0}^{\infty} \int_{B(x,r_i)\setminus B(x,r_{i+1})} \frac{f(y)}{\mu\left(B(x,r_{i+1})\right)^{1-\gamma}} \, d\mu (y) \\
%&=2 \sum_{i=0}^{\infty} \frac{1}{\mu\left(B(x,r_{i+1})\right)^{1-\gamma}} \int_{B(x,r_i)\setminus B(x,r_{i+1})} f(y) \, d\mu (y) \\
&\leq 2 \sum_{i=0}^{\infty} \frac{1}{\mu\left(B(x,r_{i+1})\right)^{1-\gamma}} \int_{B(x,r_i)} f(y) \, d\mu (y). \\
\end{align*}
Using H\"older's inequality and the doubling condition, we get
$$1 \leq C \sum_{i=0}^{\infty} \mu\left(B(x,r_{i})\right)^{\gamma-1+\frac{p-1}{p}} \Big(\int_{B(x,r_i)} f(y)^p \, d\mu (y)\Big)^{1/p}.$$
%Now raising both sides to power $p$ gives
%$$1 \leq C \sum_{i=0}^{\infty} \mu\left(B(x,r_{i})\right)^{(\gamma-1) p} \int_{B(x,r_i)} f(y)^p \, d\mu (y).$$ 
%Now since $X$ is connected, inequality \eqref{tuplaavuus toinen suunta} gives 
%\begin{align*}
%&\sum_{i=0}^{\infty} \mu\left(B(x,r_{i})\right)^{(\gamma-1) p} \int_{B(x,r_i)} f(y)^p \, d\mu (y) \\
%&=C \sum_{i=0}^{\infty} \left(\frac{\mu\left(B(x,r_{i})\right)}{\mu(B(x_0,R_0))}\right)^{(\gamma-1)p} \int_{B(x,r_i)} f(y)^p \, d\mu (y) \\
%&\leq C \sum_{i=0}^{\infty} \left(\frac{r_{i}}{R_{0}}\right)^{\left(\gamma-1\right)p s} \int_{B(x,r_i)} f(y)^p \, d\mu (y) \\
%&= C \sum_{i=0}^{\infty} 2^{-i\left(\gamma-1\right)p s} \int_{B(x,r_i)} f(y)^p \, d\mu (y)
%\end{align*}
Next, we use the fact that for any $\delta>0$ there is a constant $C>0$ such that $$1=C \sum_{i=0}^{\infty} 2^{-i \delta}=C \sum_{i=0}^{\infty} \left(\frac{r_i}{R}\right)^{\delta}.$$ 
%The geometrical series converges, when $\delta >0$, and actually $C=1-2^{-\delta}$
Then, inequality \eqref{tuplaavuus toinen suunta} gives
$$C \sum_{i=0}^{\infty} \left(\frac{\mu\left(B(x,r_{i})\right)}{\mu(B(x_0,R))}\right)^{\delta/s} \leq C \sum_{i=0}^{\infty} \left(\frac{r_i}{R}\right)^{\delta}=1.$$ Now, by putting the measure of the ball $B(x_0,R)$ as part of the constant, we have that
$$C \sum_{i=0}^{\infty} \mu(B(x,r_i))^{\delta/s} \leq \sum_{i=0}^{\infty} \mu\left(B(x,r_{i})\right)^{\gamma-1 + \frac{p-1}{p}} \Big(\int_{B(x,r_i)} f(y)^p \, d\mu (y)\Big)^{1/p},$$ where the constant $C$ depends on $R$. For $\delta>0$, there exists at least one index $i_x \in \mathbb{N}$ such that 
$$\mu\left(B(x,r_{i_x})\right)^{\gamma-1 +\frac{p-1}{p}} \Big(\int_{B(x,r_{i_x})} f(y)^p \, d\mu (y)\Big)^{1/p} \geq C \mu(B(x,r_{i_x}))^{\delta/s}$$ and, by raising both sides to the power $p$, we get
\begin{align*}
\int_{B(x,r_{i_x})} f(y)^p \, d\mu (y) &\geq C \mu(B(x,r_{i_x}))^{\delta p/s - (\gamma-1)p-p+1} \\
&= C \mu(B(x,r_{i_x}))^{\delta p/s - \gamma p+1}.
\end{align*}
%where the constant $C>0$ is independent of $f$ and $x$. 
We choose $$\delta= \frac{\gamma p - \tilde{\gamma} \tilde{p}}{p} \cdot s,$$ which is positive, as $\gamma p > \tilde{\gamma} \tilde{p}.$ We obtain for each $x \in E$ a ball $B(x, r_{i_x})=B_x$ such that
\begin{equation}\label{jepjoo}
\mu(B_x)^{1-\tilde{\gamma} \tilde{p}} \leq C \int_{B_x} f(y)^p \, d \mu (y).
\end{equation}
By using the basic $5r$-covering theorem (see e.g. \cite{He}), we obtain countably many points $x_j \in E$, such that the balls $B_j=B_{x_j}$ are pairwise disjoint and $E \subset \bigcup_{j=1}^{\infty} 5B_j$. 
%Note that the radii of the balls that cover $E$ is no more that $10 R$. 
Using the estimate \eqref{jepjoo}, the doubling property of the measure $\mu$ and the pairwise disjointness of the balls $B_j$, we get
\begin{align*}
\widetilde{\mathcal{H}}_{\infty}^{\tilde{\gamma},\tilde{p}}(E) &\leq \sum_{j=1}^{\infty} \mu(5 B_j)^{1-\tilde{\gamma} \tilde{p}} \leq C \sum_{j=1}^{\infty} \mu (B_j)^{1- \tilde{\gamma} \tilde{p}} \\
&\leq C \sum_{j=1}^{\infty} \int_{B_j} f(y)^p \, d \mu (y) \leq C \int_{X} f(y)^p \, d \mu (y) \\
&= C \, ||f||_{L^p(X)}^p< C \, \epsilon.
\end{align*}
Letting $\epsilon \rightarrow 0$ yields the claim.
\end{proof}
\end{theorem}

%\noindent Denote the Hausdorff measure, where $h(B(x_i,r_i))= \varphi_{L^{1/\alpha}}(B(x_i,r_i))=\mu(B(x_i,r_i))^{\alpha}$, $\alpha > 0$, by %$\mathcal{H}_{h_{\alpha}}$.
%
%\begin{theorem}
%Let $1<p<\infty$, $0<\gamma<1$ and $\gamma p<1$. If $C_{\gamma, p}(E)=0$, then $\mathcal{H}_{h_{\alpha}}(E)=0$ for $\alpha>1-\gamma %p$.
%\begin{proof}
%Mizuta p.175-176.
%\end{proof}
%\end{theorem}
%
%\begin{corollary}
%Let $1<p<\infty$, $0<\gamma<1$ and $\gamma p<1$. If $C_{\gamma,p}(E)=0$, then $E$ has Hausdorff dimension at most $1-\gamma p$.
%\end{corollary}

%Relations between slightly different Hausdorff content and Hausdorff measure have been studied in the metric setting, with a doubling measure $\mu$, in \cite[Section 7] %{KKST}. 
%For the next corollary, we assume $s \geq 1$ in the inequality \eqref{tuplaavuus toinen suunta} which also implies that $Q \geq 1$ in the inequality \eqref{doubling dimension}.
We can see that in the metric space, with a doubling measure $\mu$, the Hausdorff content and the Hausdorff measure have the same null sets. This relation has been studied in \cite[Section 7]{KKST} for slightly different versions of the Hausdorff content and the Hausdorff measure. For our definitions, the result of \cite[Lemma 7.6]{KKST} follows without any extra assumptions. By the previous two theorems, we get as a corollary
%as a special case the main result of \cite{Sj}. In our case, 
the following result for arbitrary sets $E$.

\begin{corollary}
Let $1<p<\infty$, $1<\tilde{p}<\infty$, $0<\gamma <1$ and $0<\tilde{\gamma}<1$ be such that $\gamma p<1$, $\tilde{\gamma} \tilde{p}<1$ and $\tilde{\gamma} \tilde{p} < \gamma p$. Then $$\mathcal{C}_{\gamma,p}(E)>0 \, \text{ implies that } \, \widetilde{\mathcal{H}}^{\gamma,p}(E)>0.$$ If we also assume our space $X$ to be connected, then $$\widetilde{\mathcal{H}}^{\tilde{\gamma},\tilde{p}}(E)>0 \, \text{ implies that } \, \mathcal{C}_{\gamma,p}(E)>0.$$
\end{corollary}

\vspace{0.5cm}
\noindent
\small{\textsc{J.N.},}
\small{\textsc{Department of Mathematics and Statistics},}
\small{\textsc{P.O. Box 35},}
\small{\textsc{FI-40014 University of Jyv\"askyl\"a},}
\small{\textsc{Finland}}\\
\footnotesize{\texttt{juho.nuutinen@jyu.fi}}

\vspace{0.3cm}
\noindent
\small{\textsc{P.S.},}
\small{\textsc{Department of Mathematics},}
\small{\textsc{P.O. Box 11100},}
\small{\textsc{FI-00076 Aalto University},}
\small{\textsc{Finland}}\\
\footnotesize{\texttt{pilar.silvestre@gmail.com}}

\end{document}